\documentclass[leqno,11pt]{amsart}

\usepackage{amsfonts}
\usepackage{amsmath}
\usepackage{amsfonts}
\usepackage{amssymb}
\usepackage{amscd}
\newtheorem{theorem}{Theorem}

\newtheorem{corollary}[theorem]{Corollary}

\newtheorem*{example*}{Example}

\newtheorem{lemma}[theorem]{Lemma}

\newtheorem{remark}[theorem]{Remark}

\begin{document}

\title{Remarks on non-compact gradient Ricci solitons}
\author{Stefano Pigola}
\address{Dipartimento di Fisica e Matematica\\
Universit\`a dell'Insubria - Como\\
via Valleggio 11\\
I-22100 Como, ITALY}
\email{stefano.pigola@uninsubria.it}

\date{\today}

\author{Michele Rimoldi}
\address{Dipartimento di Matematica\\
Universit\`a di Milano\\
via Saldini 50\\
I-20133 Milano, ITALY}
\email{michele.rimoldi@unimi.it}

\author{Alberto G. Setti}
\address{Dipartimento di Fisica e Matematica
\\
Universit\`a dell'Insubria - Como\\
via Valleggio 11\\
I-22100 Como, ITALY}
\email{alberto.setti@uninsubria.it}



\begin{abstract}
In this paper we show how techniques coming from stochastic
analysis, such as stochastic completeness (in the form of the weak
maximum principle at infinity), parabolicity  and $L^p$-Liouville type
results for the weighted Laplacian associated to the potential  may be
used to obtain triviality, rigidity results, and scalar curvature estimates
for  gradient Ricci solitons under $L^p$ conditions on the relevant
quantities.
\end{abstract}

\maketitle

\section*{Introduction}

Let $\left(  M,\left\langle ,\right\rangle \right)  $ be a Riemannian
manifold. A Ricci soliton structure on $M$ is the choice of a smooth vector
field $X$ (if any) satisfying the soliton equation%
\begin{equation}
Ric+\frac{1}{2}L_{X}\left\langle ,\right\rangle =\lambda\left\langle
,\right\rangle , \label{intro1}%
\end{equation}
for some constant $\lambda\in\mathbb{R}$. Here, $Ric$ denotes the Ricci
curvature of $M$ and $L_{X}$ stands for the Lie derivative in the direction
$X$. The Ricci soliton $\left(  M,\left\langle ,\right\rangle ,X\right)  $ is
said to be shrinking, steady or expansive according to whether the
coefficient $\lambda$ appearing in equation (\ref{intro1}) satisfies
$\lambda>0,$ $\lambda=0$ or $\lambda<0$.

In the special case where $X=\nabla f$ for some smooth function $f:M\rightarrow
\mathbb{R}$, we say that $\left(  M,\left\langle ,\right\rangle ,\nabla
f\right)  $ is a gradient Ricci soliton with potential $f$. In this situation,
the soliton equation reads%
\begin{equation}
Ric+\mathrm{Hess}\left(  f\right)  =\lambda\left\langle ,\right\rangle .
\label{intro2}%
\end{equation}

Clearly, equations (\ref{intro1}) and (\ref{intro2}) can be considered as
perturbations of the Einstein equation%
\begin{equation}
Ric=\lambda\left\langle ,\right\rangle , \label{intro3}%
\end{equation}
and reduce to this latter in case $X$ is a Killing vector field. In
particular, if $X=0,$ we call the underlying Einstein manifold a trivial Ricci soliton.

In this note we will focus our attention on geodesically complete,
gradient Ricci solitons. Here are some typical examples, \cite{PW1}.

\begin{example*}
{\rm
 The standard Euclidean space $\left(  \mathbb{R}^{m},\left\langle
,\right\rangle ,\nabla f\right)  $ with%
\[
f\left(  x\right)  =\frac{1}{2}A\left\vert x\right\vert ^{2}+\left\langle
x,B\right\rangle +C,
\]
for arbitrary $A\in \mathbb{R}$, $B\in\mathbb{R}^{m}$ and $C\in\mathbb{R}$. Note that $f$
is the essentially unique solution of the equation \textrm{Hess}$\left(
f\right)  =A \left\langle
,\right\rangle $ on $\mathbb{R}^{m}$. This follows by integrating on
$[0,\left\vert x\right\vert ]$ the equation%
\[
\frac{d^{2}}{ds^{2}}\left(  f\left(  vs\right)  \right)  =A,
\]
with $v\in\mathbb{R}^{m}$ such that $\left\vert v\right\vert =1$. In fact, a
kind of converse also holds; \cite{T}, \cite{N}, \cite{PW1}.
In the Appendix we will provide a straight-forward proof.
}
\end{example*}

\begin{theorem}
\label{th_tashiro}Let $\left(  M,\left\langle ,\right\rangle \right)  $ be a
complete manifold. Suppose that there exists a smooth function $f:M\rightarrow
\mathbb{R}$ satisfying \textrm{Hess}$\left(  f\right)  =\lambda\left\langle
,\right\rangle $, for some constant $\lambda\neq0$. Then $M$ is isometric to
$\mathbb{R}^{m}$.
\end{theorem}

\begin{example*}
{\rm  The Riemannian product%
\begin{equation}
\left(  \mathbb{R}^{m}\times N^{k},\left\langle ,\right\rangle _{\mathbb{R}%
^{m}}+\left\langle ,\right\rangle _{N^{k}},\nabla f\right)  \label{intro4}%
\end{equation}
where $\left(  N^{k},\left(  ,\right)  \right)  $ is any $k$-dimensional
Einstein manifold with Ricci curvature $\lambda\ne 0$, and $f\left(  t,x\right)
:\mathbb{R}^{m}\times N^{k}\rightarrow\mathbb{R}$ is defined by%
\begin{equation}
f\left(  x,p\right)  =\frac{\lambda}{2}\left\vert x\right\vert _{\mathbb{R}%
^{m}}^{2}+\left\langle x,B\right\rangle _{\mathbb{R}^{m}}+C, \label{intro5}%
\end{equation}
with $C\in\mathbb{R}$ and $B\in\mathbb{R}^{m}$.
}
\end{example*}

As generalizations of Einstein manifolds, Ricci solitons enjoy
some rigidity properties, which can take the form of
classification (metric rigidity), or alternatively, triviality
of the soliton structure (soliton rigidity). For instances of the former,
see e.g. the recent far-reaching paper \cite{Zhang-ArXiv}
and references therein.

As for the latter, it has been known for some time that compact,
expanding Ricci solitons are necessarily trivial, \cite{ELM}.
Our first main result, Theorem~\ref{th_expanding} below, extends this conclusion to
the non-compact setting up to imposing suitable integrability
conditions on the potential function.

Indeed, the aim of this paper is two-fold. On the one hand we
obtain triviality, rigidity results, and scalar curvature estimates for
gradient Ricci solitons under $L^p$ conditions on the relevant
quantities that extend and generalize, often in a significant way,
previous results.

On the other hand, we show how  techniques coming
from stochastic analysis, such as stochastic completeness, in the form of
the weak maximum principle at infinity, parabolicity and
$L^p$-Liouville type results for the weighted Laplacian associated
to the potential $f$, are natural in the investigations of (gradient)
Ricci solitons, and lead to elegant proofs of the above mentioned results.

\begin{theorem}
\label{th_expanding}A complete, expanding, gradient Ricci soliton \ $\left(
M,\left\langle ,\right\rangle ,\nabla f\right)  $ is trivial provided
$\ \left\vert \nabla f\right\vert \in L^{p}\left(  M,e^{-f}d\mathrm{vol}%
\right)  ,$ for some $1\leq p\leq+\infty$.
\end{theorem}

As a matter of fact, the above statement encloses three different results
according to the assumption that $p=+\infty$, $1<p<+\infty$ and $p=1$. These
will be obtained using different arguments. The $L^{\infty}$ \ situation will
be dealt with using a form of the weak maximum principle at infinity for diffusion
operators, \cite{PRS-Memoirs}, which makes an essential use of a
volume growth estimate for weighted manifolds, \cite{WW}.

This method allows, for instance, to obtain the
following estimate for the scalar curvature, which improves
results in \cite{PW1} where it is assumed that the scalar
curvature is either constant or bounded.
\begin{theorem}
\label{th_scalar_bound}
Let $(M, \langle,\,\rangle, \nabla f) $ be a geodesically complete
gradient Ricci soliton with scalar curvature $S$ and let $S_*= \inf_M
S$. If $M$ is expanding then $m\lambda \leq S_*\leq  0$; if
$M$ is shrinking then $0\leq  S_*\leq m\lambda $. Moreover,  $S_*<
m\lambda$ unless  the soliton  is trivial and $M$ is compact
Einstein, and  $S(x)>0$ on $M$ unless $S(x)\equiv 0$ on $M$,
and $M$ is isometric to $\mathbb{R}^{m}$.
\end{theorem}

On the other hand, the $L^{1<p<\infty}$ and the $L^{1}$ results will
rely on suitable Liouville-properties of the diffusion operators, \cite{PRS-vanishing},
\cite{PRS-Progress}, \cite{PRS-Memoirs}.

Further remarks on $L^{1}$-Liouville type theorems will be given in a final
section. As an application we will deduce interesting rigidity results for Ricci
solitons with integrable scalar curvature that should be compared with
\cite{PW1}, \cite{PW2}. Note that, combining Lemma 2.3 in \cite{CZ} with a volume estimate for
weighted manifolds, \cite{Mo-notices}, \cite{WW},
it follows that the scalar curvature of a shrinking Ricci soliton is $p$-integrable, for every $p>0$. We are grateful to
M. Fern\'{a}ndez-L\'{o}pez for pointing out this to us. In the expanding case we shall prove the next result. It shows that some rigidity at the end-point case $S_*=0$ in Theorem \ref{th_scalar_bound} occurs also for expanders.

\begin{theorem}
\label{th_scalar_L1}
Let $\left(  M,\left\langle ,\right\rangle ,\nabla f\right)  $ be a complete,
expanding, Ricci soliton. Let $S$ be the scalar curvature of $M$. If $S\geq0$
and $S\in L^{1}\left(  M,e^{-f}d\mathrm{vol}\right)  $ then $M$ is isometric
to the standard Euclidean space.
\end{theorem}

\section*{Acknowledgment}
The authors would like to thank M. Fern\'{a}ndez-L\'{o}pez for a careful reading of a preliminary version of the
paper and for valuable comments that led, in particular, to a substantial improvement in the case $p=1$ of Theorem \ref{th_expanding}.
\section{Basic equations}

The geometric quantities related to gradient Ricci solitons satisfy a
number of differential identities that have been explored in \ several papers. We are
interested in the elliptic point of view, therefore we limit ourselves to
quoting the interesting papers \cite{ELM} and \cite{PW1}, \cite{PW2}, which are
particularly relevant to our investigation. Following the notation introduced
in \cite{PW1}, \cite{PW2}, we set%
\begin{equation}
\Delta_{f}u=e^{f}\operatorname{div}\left(  e^{-f}\nabla u\right)  .
\label{basic1}%
\end{equation}
In the next sections we will use the following Bochner-type identities.

\begin{lemma}
\label{lemma_basiceq}Let $\left(  M,\left\langle ,\right\rangle ,\nabla
f\right)  $ be a gradient Ricci soliton. Then%
\begin{equation}
\frac{1}{2}\Delta\left\vert \nabla f\right\vert ^{2}=\left\vert \mathrm{Hess}%
\left(  f\right)  \right\vert ^{2}-Ric\left(  \nabla f,\nabla f\right)
\label{basic2}%
\end{equation}
and%
\begin{equation}
\frac{1}{2}\Delta_{f}\left\vert \nabla f\right\vert ^{2}=\left\vert
\mathrm{Hess}\left(  f\right)  \right\vert ^{2}-\lambda\left\vert \nabla
f\right\vert ^{2}, \label{basic3}%
\end{equation}
where $\lambda$ is defined in (\ref{intro2}).
\end{lemma}

In particular, combining Lemma \ref{lemma_basiceq} with the Kato inequality%
\begin{equation}
\left\vert \mathrm{Hess}\left(  f\right)  \right\vert ^{2}\geq\left\vert
\nabla\left\vert \nabla f\right\vert \right\vert ^{2}, \label{basic4}%
\end{equation}
we deduce the next

\begin{corollary}
\label{cor_basicineq}Let $\left(  M,\left\langle ,\right\rangle ,\nabla
f\right)  $ be a gradient Ricci soliton. Then, $\left\vert \nabla f\right\vert
\in Lip_{loc}\left(  M\right)  $ satisfies%
\begin{equation}
\left\vert \nabla f\right\vert \Delta\left\vert \nabla f\right\vert
\geq-Ric\left(  \nabla f,\nabla f\right)  \label{basic5}%
\end{equation}
weakly on $M$  and%
\begin{equation}
\left\vert \nabla f\right\vert \Delta_{f}\left\vert \nabla f\right\vert
\geq-\lambda\left\vert \nabla f\right\vert ^{2}, \label{basic6}%
\end{equation}
weakly on $\left(  M,e^{-f}d\mathrm{vol}\right)  $.
\end{corollary}

Thus, not surprisingly, from the Bochner equation viewpoint, the vector field
$X=\nabla f$ behaves like a Killing field. Therefore, the standard Bochner
technique implies that if $\left(  M,\left\langle ,\right\rangle ,\nabla
f\right)  $ is a compact Ricci soliton with $Ric\leq0$ then $f$ must be constant
and, hence, $M$ is Einstein. Similar conclusions can be obtained in the non-compact
setting using global forms of the Stokes theorem. In fact, a little amount of positive Ricci
curvature is also allowed as explained in \cite{PRS-Progress}.

We shall also use the next computations
concerning the scalar curvature of a gradient Ricci soliton; \cite{ELM},
\cite{PW1}.

\begin{theorem}
\label{th_scalar}Let $\left(  M,\left\langle ,\right\rangle ,\nabla f\right)
$ be a gradient Ricci soliton with scalar curvature $S$ and Ricci curvature
$Ric.$ Then%
\begin{equation}
\Delta_{f}S=\lambda S-\left\vert Ric\right\vert ^{2}. \label{basic7}%
\end{equation}

\end{theorem}

\section{Triviality of expanders under $L^{\infty}$ conditions
 and scalar curvature estimates}

It is known, \cite{LR-Annalen}, that a complete, shrinking Ricci soliton
$\left(  M,\left\langle ,\right\rangle ,X\right)  $ satisfying $\left\vert
X\right\vert \in L^{\infty}$ must be compact. In this section we show that, in
case the soliton is gradient and expanding, the $L^{\infty}$ condition implies
triviality. To simplify the writings, having fixed a smooth function
$f:M\rightarrow\mathbb{R}$, we denote%
\begin{equation}
Ric_{f}=Ric+\mathrm{Hess}\left(  f\right)  \label{bounded1}%
\end{equation}
which is called the Bakry-Emery Ricci tensor of the weighted manifold%
\begin{equation}
\left(  M,\left\langle ,\right\rangle ,e^{-f}d\mathrm{vol}\right)
.\label{bounded1a}%
\end{equation}
Thus, $\left(  M,\left\langle ,\right\rangle ,\nabla f\right)  $ is a Ricci
soliton provided the corresponding weighted manifold has constant $Ric_{f}%
$-curvature, i.e.,%
\begin{equation}
Ric_{f}\equiv\lambda,\label{bounded2}%
\end{equation}
for some $\lambda\in\mathbb{R}$. If $B_{r}\left(  p\right)  $ and $\partial
B_{r}\left(  p\right)  $ denotes respectively the metric ball and sphere of
$\left(  M,\left\langle ,\right\rangle \right)  $ of radius $r>0$ and centered
at $p\in M$, we also set%
\[
\mathrm{vol}_{f}\left(  B_{r}\left(  p\right)  \right)  =\int_{B_{r}\left(
p\right)  }e^{-f}d\mathrm{vol},\qquad\mathrm{vol}_{f}\left(  \partial
B_{r}\left(  p\right)  \right)  =\int_{\partial B_{r}\left(  p\right)  }%
e^{-f}d\mathrm{vol}_{m-1},
\]
where $d\mathrm{vol}_{m-1}$ stands for the $\left(  m-1\right)  $-Hausdorff
measure. In the previous section, we have also introduced the second order,
diffusion operator%
\begin{equation}
\Delta_{f}u=e^{f}\operatorname{div}\left(  e^{-f}\nabla u\right)
,\label{bounded4}%
\end{equation}
which is formally self-adjoint in $L^{2}\left(  M,e^{-f}d\mathrm{vol}\right)
$. For the sake of convenience we call $\Delta_{f}$ the $f$-Laplacian.

In a way similar, but by no means equal, to the (Riemannian) non-weighted case
$f=\mathrm{const.}$, there are mutual relations between $Ric_{f}$-bounds,
\textrm{vol}$_{f}$-growth properties of metric balls and the analysis and
geometry of $\Delta_{f}$. In view of our purposes, we shall limit ourselves to
quoting the following two results. First, we recall a weighted-volume comparison
established in \cite{WW}, Theorem 3.1.

\begin{theorem}
\label{th_comparison}Let $\left(  M,\left\langle ,\right\rangle ,e^{-f}%
d\mathrm{vol}\right)  $ be a geodesically complete weighted manifold. Suppose
that%
\begin{equation}
Ric_{f}\geq\lambda, \label{bounded5}%
\end{equation}
for some constant $\lambda\in\mathbb{R}$. Then, having fixed $R_{0}>0$, there
are constants $A,B,C>0$ such that, for every $r\geq R_{0}$,%
\begin{equation}
\mathrm{vol}_{f}\left(  B_{r}\right)  \leq A+B\int_{R_{0}}^{r}e^{-\lambda
t^{2}+Ct}dt. \label{bounded6}%
\end{equation}

\end{theorem}

We recall that if $\left(  M,\left\langle ,\right\rangle ,e^{-f}
d\mathrm{vol}\right)$ is a weighted manifold, we say that the weak maximum
principle at infinity for $\Delta_f$ holds if
given a $C^{2}$ function $u:M\rightarrow\mathbb{R}$ satisfying
$\sup_{M}u=u^{\ast}<+\infty$, there exists a sequence $\left\{  x_{n}\right\}
\subset M$ along which%
\begin{equation}
\label{bounded8}
(i)\,\, u\left(  x_{n}\right)  \geq u^{\ast}-\frac{1}{n}
\quad \text{ and } \quad
(ii)\,\, \Delta_{f}u\left(  x_{n}\right)  \leq\frac{1}{n}.
\end{equation}

The next result states the validity of a weak form of the maximum principle at
infinity for the $f$-Laplacian, under weighted volume growth conditions. It
can be deduced from \cite{PRS-Memoirs} Theorem 3.11, making minor modifications
in  the proofs of Lemma 3.13, Lemma 3.14, Theorem 3.15 and Corollary 3.16.

\begin{theorem}
\label{th_maximum}Let $\left(  M,\left\langle ,\right\rangle ,e^{-f}%
d\mathrm{vol}\right)  $ be a geodesically complete weighted manifold
satisfying the volume growth condition%
\begin{equation}
\frac{r}{\log \mathrm{vol}_{f}\left(  B_{r}\right)  }\notin L^{1}\left(  +\infty
\right)  . \label{bounded7}%
\end{equation}
Then, the weak maximum principle at infinity for the $f$-Laplacian holds on
$M$.
\end{theorem}

Combining Theorems \ref{th_comparison} and \ref{th_maximum} immediately gives
the following

\begin{corollary}
\label{cor_maximum}Let $\left(  M,\left\langle ,\right\rangle ,\nabla
f\right)  $ be a geodesically complete Ricci soliton which is either
shrinking, steady or expanding. Then, the weak maximum principle at infinity
for the $f$-Laplacian holds on $M$.
\end{corollary}

We are now in the position to prove the first main result of the paper.

\begin{theorem}
\label{th_trivexp1}Let $\left(  M,\left\langle ,\right\rangle ,\nabla
f\right)  $ be a geodesically complete, expanding Ricci soliton with $\sup
_{M}\left\vert \nabla f\right\vert <+\infty$. Then the Ricci soliton is trivial.
\end{theorem}

\begin{proof}
According to (\ref{basic3}) the smooth function $\left\vert \nabla
f\right\vert ^{2}$ satisfies%
\begin{equation}
\frac{1}{2}\Delta_{f}\left\vert \nabla f\right\vert ^{2}\geq-\lambda\left\vert
\nabla f\right\vert ^{2}\geq0. \label{bounded10}%
\end{equation}
Applying Corollary \ref{cor_maximum} we deduce that there exists a sequence
$\left\{  x_{n}\right\}  \subset M$ such that,
\begin{equation}
\left\vert \nabla f\right\vert ^{2}\left(  x_{n}\right)  \geq\sup
_{M}\left\vert \nabla f\right\vert^2 -\frac{1}{n}, \label{bounded11}%
\end{equation}
and%
\begin{equation}
\Delta_{f}\left\vert \nabla f\right\vert ^{2}\left(  x_{n}\right)  \leq
\frac{1}{n}. \label{bounded12}%
\end{equation}
Evaluating (\ref{bounded10}) along $\left\{  x_{n}\right\}  $ and taking the
limit as $n\rightarrow+\infty$ we conclude%
\[
-\lambda\sup_{M}\left\vert \nabla f\right\vert ^{2}=0,
\]
proving that $f$ is constant.
\end{proof}

The estimate on the scalar curvature in Theorem~\ref{th_scalar_bound} follows
now by combining Corollary~\ref{cor_maximum} with the following
\textquotedblleft a-priori\textquotedblright\ estimate for weak solutions of
semi-linear elliptic inequalities under volume assumptions. It is an
adaptation of Theorem B in \cite{PRS-gafa}.

\begin{theorem}
\label{th_apriori}Let $\left(  M,\left\langle ,\right\rangle ,e^{-f}%
d\mathrm{vol}\right)  $ be a complete, weighted manifold. Let $a\left(
x\right)  ,$ $b\left(  x\right)  \in C^{0}\left(  M\right)  $, set
$a_{-}\left(  x\right)  =\max\left\{  -a\left(  x\right)  ,0\right\}  $ and
assume that
$$\sup_{M}a_{-}\left(  x\right)  <+\infty$$
and
\[
b\left(  x\right)  \geq\frac{1}{Q\left(  r\left(  x\right)  \right)  }\text{
on }M,
\]
for some positive, non-decreasing function $Q\left(  t\right)  $ such that
$Q\left(  t\right)  =o\left(  t^{2}\right)  $, as $t\rightarrow +\infty$. Assume
furthermore that, for some $H>0$,%
\[
\frac{a_{-}\left(  x\right)  }{b\left(  x\right)  }\leq H\text{, on }M.
\]
Let $u\in Lip_{loc}\left(  M\right)  $ be a non-negative solution of%
\begin{equation}
\Delta_{f}u\geq a\left(  x\right)  u+b\left(  x\right)  u^{\sigma}%
\text{,}\label{a-priori3}%
\end{equation}
weakly on $\left(  M,e^{-f}d\mathrm{vol}\right)  $, with $\sigma>1$. If%
\begin{equation}
\liminf_{r\rightarrow+\infty}\frac{Q\left(  r\right)  \log\mathrm{vol}%
_{f}\left(  B_{r}\right)  }{r^{2}}<+\infty,\label{a-priori4}%
\end{equation}
then%
\[
u\left(  x\right)  \leq H^{\frac{1}{\sigma-1}}\text{, on }M.
\]

\end{theorem}

\begin{proof}
We have only to verify that the integral inequality stated in Lemma 1.5 on
page 1309 of \cite{PRS-gafa} holds with respect to the weighted measure
$e^{-f}d\mathrm{vol}$. This in turn can be deduced exactly as in
\cite{PRS-gafa} \ provided (the weighted version of) inequality (1.21) on page
1310 is satisfied. Now, by assumption, for every compactly supported $\rho\in
W_{loc}^{1,2}\left(  M,e^{-f}d\mathrm{vol}\right)  $, $\rho\geq0$, we have%
\[
-\int\left\langle \nabla u,\nabla\rho\right\rangle e^{-f}d\mathrm{vol}\geq
\int\left(  au\rho+bu^{\sigma}\rho\right)  e^{-f}d\mathrm{vol}.
\]
Therefore, the desired inequality (1.21) follows by taking%
\[
\rho=\lambda\left(  u\right)  \psi^{2\left(  \alpha+\sigma-1\right)
}u^{\alpha-1}%
\]
with $\alpha\geq2$.
\end{proof}

Using Theorem \ref{th_comparison} we deduce the validity of the next

\begin{corollary}
\label{cor_apriori}Let $\left(  M,\left\langle ,\right\rangle ,\nabla
f\right)  $ be a complete Ricci soliton and let $u\in Lip_{loc}\left(
M\right)  $ be a non-negative weak solution of
\[
\Delta_{f}u\geq au+bu^{\sigma}\text{,}%
\]
for some constants $a\in\mathbb{R}$, $b>0$ and $\sigma>1$. Then%
\[
u\left(  x\right)^{\sigma-1}  \leq\frac{\max\left\{  -a,0\right\}  }{b}.
\]

\end{corollary}

We are now in the position to give the

\begin{proof}
[Proof of Theorem~\ref{th_scalar_bound}]By the Cauchy-Schwarz inequality,
$|\mathrm{Ric}|^{2}\geq\frac{1}{m}S^{2}$ and inserting in (\ref{basic7}) we
deduce that
\begin{equation}
\Delta_{f}S\leq\lambda S-\frac{1}{m}S^{2}. \label{bounded13}%
\end{equation}
In follows that $S_{-}\left(  x\right)  =\max\left\{  -S\left(  x\right)
,0\right\}  $ is a weak solution of%
\[
\Delta_{f}S_{-}\geq\lambda S_{-}+\frac{1}{m}S_{-}^{2}.
\]
Therefore, by Corollary \ref{cor_apriori}, $S_{-}$ is bounded from above or,
equivalently, $S_{\ast}=\inf_{M}S>-\infty$ (for this conclusion, see also
\cite{Zhang-PAMS}).
Applying Corollary \ref{cor_maximum} produces a sequence $\{x_{n}\}$ such that $\Delta_{f}%
S(x_{n})\geq-1/n$ and $S(x_{n})\rightarrow S_{\ast}$, and taking the liminf in
(\ref{bounded13}) along $\{x_{n}\}$ shows that $\lambda S_{\ast}-S_{\ast}%
^{2}/m\geq0$. Thus, if $\lambda<0,$ then $m\lambda\leq S_{\ast}\leq0,$ while,
if $\lambda>0$, then $0\leq S_{\ast}\leq m\lambda.$

Assume now that $S_{\ast}=\lambda m>0$. Then $S\geq S_{\ast}=m\lambda$ and
$\lambda S-\frac{1}{m}S^{2}\leq0$. It follows from (\ref{bounded13}) that
$S>0$ satisfies $\Delta_{f}S\leq0$. By Theorem~\ref{th_solitonparabolic} a
supersolution of $\triangle_{f}$ which is bounded below is constant. Hence,
$S=S_{\ast}=m\lambda$ is a constant, and $|\mathrm{Ric}|^{2}=\frac{1}{m}S^{2}%
$. By the equality case in the Cauchy-Schwarz inequality, we deduce that
$\mathrm{Ric}=\lambda\langle\,,\rangle$ with $\lambda>0$ and $M$ is compact by
Myers' Theorem. By (\ref{intro2}) $\mathrm{Hess}(f)=0$, and in particular $f$
is a harmonic function on $M$ compact, and therefore it is
constant.

Finally, since $S(x)\geq 0$, by the maximum principle (see \cite{GilbargTrudinger}, p. 35),
either $S(x)>0$ on $M$ or  $S(x)\equiv 0$. In the latter case it
follows from  (\ref{basic7}) that $Ric\equiv 0$ and then, by soliton
equation,
we conclude that  $f$ is a (necessarily non trivial) solution of%
\[
\mathrm{Hess}\left(  f\right)  =\lambda\left\langle\, ,\right\rangle .
\]
By Theorem \ref{th_tashiro} stated in the Introduction, $\left(
M,\left\langle \,,\right\rangle \right)  $ is isometric to $\mathbb{R}^{m}$.
\end{proof}

\section{Triviality of expanders under $L^{1<p<\infty}$ conditions}

It is well known that a non-negative, $L^{p}$-subharmonic function,
$1<p<+\infty$, on a complete Riemannian manifold must be constant, \cite{Y}.
This classical Liouville type theorem has been extended in various directions
to both linear and non-linear operators. Here we recall the following version
for the $f$-Laplacian established in \cite{PRS-vanishing}, Theorem 1.1. See
also \cite{PRS-Progress}. Recently, somewhat less general forms of this result
have been independently rediscovered in \cite{N}, \cite{PW1}, \cite{PW2}.

\begin{theorem}
\label{th_liouville}Let $\left(  M,\left\langle ,\right\rangle ,e^{-f}%
d\mathrm{vol}\right)  $ be a geodesically complete weighted manifold. Assume
that $u\in Lip_{loc}\left(  M\right)  $ satisfy%
\begin{equation}
u\Delta_{f}u\geq0\text{, weakly on } (  M,e^{-f}d\mathrm{vol}). \label{lp1}%
\end{equation}
If, for some $p>1$,%
\begin{equation}
\frac{1}{\int_{\partial B_{r}}\left\vert u\right\vert ^{p}e^{-f}d\mathrm{vol}_{m-1}%
}\notin L^{1}\left(  +\infty\right)  ,\label{lp2}%
\end{equation}
then $u$ is constant.
\end{theorem}

\begin{remark}
\label{rem_positivepart}
{\rm Observe that if $u\in L^{p}\left(  M,e^{-f}%
d\mathrm{vol}\right)  $ then condition (\ref{lp2}) is satisfied. Note also
that no sign condition is required on $u$. Moreover, if the locally lipschitz
function $u$ satisfies both $\Delta_{f}u\geq0$ and the non-integrability
condition (\ref{lp2}) then, applying Theorem \ref{th_liouville} to $u_{+}%
=\max\left\{  u,0\right\}  $, gives that either $u$ is constant or $u\leq0$.
}
\end{remark}

\begin{theorem}
\label{th_trivexp2}Let $\left(  M,\left\langle ,\right\rangle ,\nabla
f\right)  $ be a geodesically complete, expanding Ricci soliton. If
\[
\frac{1}{\int_{\partial B_{r}}\left\vert \nabla f\right\vert ^{p}%
e^{-f}d\mathrm{vol}}_{m-1}\notin L^{1}\left(  +\infty\right)  ,
\]
for some $p>1$ then the soliton is trivial.
\end{theorem}

\begin{proof}
Recall from equation (\ref{basic6}) that%
\[
\left\vert \nabla f\right\vert \Delta_{f}\left\vert \nabla f\right\vert
\geq-\lambda\left\vert \nabla f\right\vert ^{2}\geq0\text{, weakly on }(  M,e^{-f}d\mathrm{vol}).
\]
An application of Theorem \ref{th_liouville} gives that $\left\vert \nabla
f\right\vert $ is constant. Using this information into (\ref{basic3}) we
conclude that $\left\vert \nabla f\right\vert =0$ and $f$ is a constant function.
\end{proof}

\section{Triviality of expanders under $L^{1}$ conditions}

The following result has been recently obtained in \cite{PRS-Progress},
Theorem 4.3.

\begin{theorem}
\label{th_rs-revista}Let $\left(  M,\left\langle ,\right\rangle ,e^{-f}%
d\mathrm{vol}\right)  $ be a geodesically complete weighted manifold. Let
$0\leq u\in Lip_{loc}\left(  M\right)  $ be a weak solution of $\Delta
_{f}u\geq0$ satisfying%
\[
\text{(i) }\int_{\partial B_{r}}ue^{-f}d\mathrm{vol}_{m-1}\left(  x\right)
=O\left(  \frac{1}{r\log^{\alpha}r}\right)  ,\qquad\text{(ii) }u\left(
x\right)  =O\left(  e^{\beta r\left(  x\right)  ^{2}}\right)  ,
\]
as $r\left(  x\right)  \rightarrow+\infty$, for some constants $\alpha
,\beta>0.$ Then $u$ is constant.
\end{theorem}
Note that although  Theorem 4.3 is stated with $\beta=1$ in condition (ii),
the proof shows that  the more general version stated above holds.

In particular, applying the theorem to the positive part $u_{+}=\max\left\{
u,0\right\}  $ of the given solution $u$ yields the following

\begin{corollary}
Let $\left(  M,\left\langle ,\right\rangle ,e^{-f}d\mathrm{vol}\right)  $ be a
geodesically complete weighted manifold. If $u\in Lip_{loc}\left(  M\right)
\cap L^{1}\left(  M,e^{-f}d\mathrm{vol}\right)  $ is a solution of $\Delta
_{f}u\geq0$ satisfying $u\left(  x\right)  \leq\alpha e^{\beta r\left(
x\right)  ^{2}}$, for some constants $\alpha,\beta>0$, then either $u$ is
constant or $u\leq0$.
\end{corollary}

In order to apply Theorem \ref{th_rs-revista} and conclude triviality of
expanders under solely $L^1$ conditions we also need the following estimate from \cite{Zhang-PAMS}.

\begin{theorem}\label{lo-estimate}
Let $(M,\left<,\right>,\nabla f)$ be a complete, expanding Ricci soliton. Then, having fixed a reference origin $o \in M$, there exists a constant $c>0$ such that
\begin{align*}
&(1)\text{ } f(x)\leq c(1+r(x)^2),\\&(2)\text{ }|\nabla f|\leq c(1 + r(x)).
\end{align*}
\end{theorem}

\begin{remark}
\rm{Note that, according to the scalar curvature estimates of Theorem \ref{th_scalar_bound}, the above constant $c>0$ can be expressed in terms of the soliton constant $\lambda<0$ and the dimension of $M$.}
\end{remark}

As an immediate consequence of Theorems \ref{th_rs-revista} and \ref{lo-estimate}, arguing as in the
proof of Theorem \ref{th_trivexp2}, we get the next

\begin{theorem}
Let $\left(  M,\left\langle ,\right\rangle ,\nabla f\right)  $ be a
geodesically complete, expanding Ricci soliton. If%
\begin{equation}
\int_{\partial B_{r}}\left\vert \nabla f\right\vert e^{-f}%
d\mathrm{vol}_{m-1}=O\left(  \frac{1}{r\log^{\alpha}r}\right)
, \label{rs-revista2}%
\end{equation}
for some positive constants $\alpha,\beta$, and for $r\left(  x\right)  $
sufficiently large, then the soliton is trivial.
\end{theorem}

\section{More on $L^{1}$-Liouville theorems and some rigidity results}

Following classical terminology in linear potential theory we say that \ a
weighted Riemannian manifold $\left(  M,\left\langle ,\right\rangle
,e^{-f}d\mathrm{vol}\right)  $ is $f$-parabolic if every solution of
$\Delta_{f}u\geq0$ satisfying $u^{\ast}=\sup_{M}u<+\infty$ must be constant.
Equivalently, $\left(  M,\left\langle ,\right\rangle ,e^{-f}d\mathrm{vol}%
\right)  $ is non-parabolic if and only if $\Delta_{f}$ possesses a positive,
minimal Green kernel $G_{f}\left(  x,y\right)  $. It can be shown that a
sufficient condition for $\left(  M,\left\langle ,\right\rangle ,e^{-f}%
d\mathrm{vol}\right)  $ to be parabolic is that $M$ is geodesically complete
and%
\begin{equation}
\mathrm{vol}_{f}\left(  \partial B_{r}\right)  ^{-1}\notin L^{1}\left(
+\infty\right)  . \label{more1}%
\end{equation}
All these facts can be easily established adapting to the diffusion operator
$\Delta_{f}$ standard proofs for the Laplace-Beltrami operator; \cite{G},
\cite{RS}. In particular, according to Theorem \ref{th_comparison} we have

\begin{theorem}
\label{th_solitonparabolic}A complete, gradient shrinking Ricci soliton
$\left(  M,\left\langle ,\right\rangle ,\nabla f\right)  $ is $f$-parabolic.
\end{theorem}

We also point out the following consequence of Theorem
\ref{th_solitonparabolic}, Theorem \ref{th_liouville} and Remark
\ref{rem_positivepart}.

\begin{corollary}
Let $\left(  M,\left\langle ,\right\rangle ,\nabla f\right)  $ \ be a
complete, gradient, shrinking Ricci soliton. If $u\in Lip_{loc}\left(
M\right)  $ satisfies $\Delta_{f}u\geq0$ and $u\in L^{p}\left(  M,e^{-f}%
d\mathrm{vol}\right)  $, for some $1<p<+\infty$, then $u$ is constant.
\end{corollary}

It can be shown that $f$-parabolicity implies the validity of the weak maximum
principle at infinity for the operator $\Delta_{f}$.
This follows in a way similar to the case $f=0$,  noting that the
weak maximum principle is equivalent to the property that if $u$
is a non-negative bounded function satisfying $\Delta_f u \geq \mu
u$ for some $\mu >0$ then $u\equiv 0$ (see \cite{PRS-Memoirs},
Theorem 3.11).

In a different direction, the diffusion operator $\Delta_{f}$ has a minimal, positive heat
kernel $p_f\left(  t,x,y\right)$ and the validity of the weak maximum principle at infinity is
also equivalent to the property%
\begin{equation}
\int_{M}p_{f}\left(  t,x,y\right)  e^{-f}d\mathrm{vol}\left(  y\right)
=1,\label{more2}%
\end{equation}
for every $t>0$ and for every $x\in M$, \cite{PRS-Memoirs}.

>From a probabilistic viewpoint, condition (\ref{more2}) states that the
diffusion process with transition probabilities $p_{f}\left(  t,x,y\right)  $
is Markovian, hence stochastically complete. In case $f\equiv0$, it is known
that stochastic completeness with respect to the Brownian motion on $\left(
M,\left\langle ,\right\rangle \right)  $ is related to $L^{1}$ Liouville type
properties for super-harmonic functions, \cite{G2}.

Rephrasing \ these properties for the operator $\Delta_{f}$, we say that
 the $L^1$ Liouville property for  $\Delta_f$-superharmonic functions holds if
every
$Lip_{loc}$ solution  of $\Delta_{f}u\leq0$ satisfying $0\leq u\in
L^{1}\left(  M,e^{-f}d\mathrm{vol}\right)  $ \ must be constant.

Using exactly the same proof as in the case $f\equiv0$, \cite{G2}, shows
that this is equivalent to the fact that for some, hence for all, $x\in
M$,
\begin{equation}
\label{Green_not_L1} \int_M G_f(x,y) \, e^{-f} d\mathrm{vol}(y) =
+\infty.
\end{equation}
Recalling that the Green kernel $G_{f}$
is related to the heat kernel $p_{f}$ by the formula%
\begin{equation}
G_{f}\left(  x,y\right)  =\int_{0}^{+\infty}p_{f}\left(  t,x,y\right)
dt,\label{more3},
\end{equation}
from the above circle of ideas one obtains
\begin{theorem}
If the weak maximum principle at infinity holds for $\Delta_{f}$
then the $L^1$ Liouville property for  $\Delta_f$-superharmonic functions holds.
\end{theorem}

In particular, combining with Theorem \ref{th_maximum} we conclude the
validity of the next Liouville type property of Ricci soliton.

\begin{theorem}
\label{th_l1-superharmonic}Let $\left(  M,\left\langle ,\right\rangle ,\nabla
f\right)  $ be a complete, gradient Ricci soliton. Then the
$L^1$ Liouville property for  $\Delta_f$-superharmonic functions holds.
\end{theorem}

\begin{remark}
{\rm
Since, by Theorem \ref{th_solitonparabolic},  shrinking solitons are $f$-parabolic,
in this situation the same conclusion holds without any integrability assumption on $u$.
}
\end{remark}

By way of example, we now apply this result to prove the  rigidity of gradient Ricci
solitons with integrable scalar curvature stated in Theorem~\ref{th_scalar_L1}.

\begin{proof} [Proof of Theorem~\ref{th_scalar_L1}]
Recall that, by formula (\ref{basic7}) of Theorem \ref{th_scalar}, it holds%
\begin{equation}
\Delta_{f}S=\lambda S-\left\vert Ric\right\vert ^{2}, \label{more4}%
\end{equation}
where $\lambda<0$ is such that%
\begin{equation}
Ric+\mathrm{Hess}\left(  f\right)  =\lambda\left\langle ,\right\rangle .
\label{more5}%
\end{equation}
Since $S\geq0$, from the above we deduce%
\begin{equation}
\Delta_{f}S\leq0. \label{more6}%
\end{equation}
Applying Theorem \ref{th_l1-superharmonic} we obtain that $S$ is constant.
Using this information into (\ref{more4}) implies that $Ric\equiv0$,
and the required conclusion follows from Theorem~\ref{th_tashiro} as
in the last part of the proof of Theorem~\ref{th_scalar_bound}
\end{proof}

\section*{Appendix}

In this section we provide a somewhat detailed proof of Theorem
\ref{th_tashiro}. Our basic reference for Riemannian geometry is \cite{Pe}.
Notation is that introduced there. Note that our proof generalizes to give
a characterization of general model manifolds via second (and third) order
differential systems, \cite{PR}.\medskip

We shall use the following density result, \cite{B}, \cite{Wo}. Following
Bishop, recall that, given a complete manifold $\left(  M,\left\langle
,\right\rangle \right)  $ and a reference point $o\in M$, then $p\in
cut\left(  o\right)  $ is an\textit{ ordinary cut point }if there are at least
two distinct minimizing geodesics from $o$ to $p$. Using the infinitesimal
Euclidean law of cosines, it is not difficult to show that at an ordinary cut
point $p$ the distance function $r\left(  x\right)  =\mathrm{dist}_{\left(
M,\left\langle ,\right\rangle \right)  }\left(  x,o\right)  $ is not
differentiable, \cite{Wo}.

\begin{theorem}
[Bishop density result]\label{th_bishop}Let $\left(  M,\left\langle
,\right\rangle \right)  $ be a complete Riemannian manifold and let $o\in M$
be a reference point. Then the ordinary cut-points of $o$ are dense in
$cut\left(  o\right)  $. In particular, if the distance function $r\left(
x\right)  $ from $o$ is differentiable on the (punctured) open ball
$B_{R}\left(  o\right)  \backslash\left\{  o\right\}  $ then $B_{R}\left(
o\right)  \cap cut\left(  o\right)  =\emptyset$.
\end{theorem}

Now, let $f\in C^{\infty}\left(  M\right)  $ be a solution of%
\begin{equation}
\mathrm{Hess}\left(  f\right)  =\lambda\left\langle ,\right\rangle ,
\label{appendix1}%
\end{equation}
for some constant $\lambda\neq0$. Without loss of generality, we assume
$\lambda>0$. To simplify the exposition we proceed by steps.\medskip

\textit{Step 1.} We note first that $f$ has a critical point. Indeed,
by contradiction, suppose $\left\vert \nabla f\right\vert \neq0$ on $M$.
Consider the vector field $X=\nabla f/\left\vert \nabla f\right\vert $ on $M$.
Clearly, $X$ is complete because $\left\vert X\right\vert \in L^{\infty
}\left(  M\right)  $ and $\left(  M,\left\langle ,\right\rangle \right)  $ is
geodesically complete. Let $\gamma:\mathbb{R}\rightarrow M$ be an integral
curve of $X$, i.e., $X_{\gamma}=\dot{\gamma}$. It is readily verified from
equation (\ref{appendix1}) that, for every vector field $Y$,%
\begin{equation}
\left\langle D_{\dot{\gamma}}\dot{\gamma},Y\right\rangle =\frac{1}{\left\vert
\nabla f\right\vert }\mathrm{Hess}\left(  f\right)  \left(  \dot{\gamma
},Y\right)  - \frac{1}{\left\vert
\nabla f\right\vert } \mathrm{Hess}\left(  f\right)  \left(  \dot{\gamma
},\dot{\gamma
} \right)\left\langle \dot{\gamma},Y\right\rangle =0.
\label{appendix2}%
\end{equation}
Therefore, $\gamma$ is a geodesic. Evaluating (\ref{appendix1}) along $\gamma$
we deduce that the smooth function $y\left(  t\right)  =f\circ\gamma\left(
t\right)  $ satisfies%
\[
\frac{d^{2}y}{dt^{2}}=\lambda.
\]
Integrating on $[0,t]$ yields $y^{\prime}\left(  t\right)  =\lambda
t+y^{\prime}\left(  0\right)  ,$ so that $y^{\prime}\left(  t_{0}\right)  =0$
where $t_{0}=-\lambda^{-1}y^{\prime}\left(  0\right)  $. It follows that%
\begin{equation}
0=y^{\prime}\left(  t_{0}\right)  =\left\langle \nabla f\left(  \gamma\left(
t_{0}\right)  \right)  ,\dot{\gamma}\left(  t_{0}\right)  \right\rangle
=\left\vert \nabla f\right\vert \left(  t_{0}\right)  \neq0, \label{appendix5}%
\end{equation}
contradiction.\medskip

\textit{Step 2.} Let $o\in M$ \ be a critical point of $\nabla f$ and set
$r\left(  x\right)  =\mathrm{dist}_{\left(  M,\left\langle ,\right\rangle
\right)  }\left(  x,o\right)  $. Having fixed $x\in M$, let $\gamma
:[0,r\left(  x\right)  ]\rightarrow M$ be a unit speed, minimizing geodesic
issuing from $\gamma\left(  0\right)  =o$. Therefore, $y\left(  t\right)
=f\circ\gamma\left(  t\right)  $ solves the Cauchy problem%
\begin{equation}
\left\{
\begin{array}
[c]{l}%
\dfrac{d^{2}y}{dt^{2}}=\lambda\\
y^{\prime}\left(  0\right)  =0,\text{ }y\left(  0\right)  =f\left(  o\right)
.
\end{array}
\right.  \label{appendix6}%
\end{equation}
Integrating on $[0,r\left(  x\right)  ]$ we deduce that%
\begin{equation}
f\left(  x\right)  =\alpha\left(  r\left(  x\right)  \right)  ,
\label{appendix7}%
\end{equation}
where%
\begin{equation}
\alpha\left(  t\right)  =\frac{\lambda}{2}t^{2}+f\left(  o\right)  .
\label{appendix8}%
\end{equation}
In particular, $f$ is a proper function with precisely one critical
point.\medskip

\textit{Step 3. }\ Since $f\left(  x\right)  =\alpha\left(  r\left(  x\right)
\right)  $ is smooth and $\alpha\left(  t\right)  $ satisfies $\alpha^{\prime
}\left(  t\right)  \neq0$ for every $t>0$ , it follows that%
\begin{equation}
r\left(  x\right)  =\alpha^{-1}\left(  f\left(  x\right)  \right)
\label{appendix9}%
\end{equation}
is smooth on $M\backslash\left\{  o\right\}  .$ According to Theorem
\ref{th_bishop} we have $cut\left(  o\right)  =\emptyset$ and the exponential
map $\exp_{o}:T_{o}M\approx\mathbb{R}^{m}\rightarrow M$ realizes a smooth
diffeomorphism. Let us introduce geodesic polar coordinates $\left(
r,\theta\right)  \in(0,+\infty)\times S^{m-1}$ on $T_{o}M$. Moreover, let us
consider a local orthonormal frame $\left\{  E_{\alpha}\right\}  $ on
$S^{m-1}$ with dual frame $\left\{  \theta^{\alpha}\right\}  $ \ and extend
them radially. Then, by Gauss lemma,%
\begin{equation}
\left\langle ,\right\rangle =dr\otimes dr+\sum_{\alpha,\beta=1}^{m-1}%
\sigma_{\alpha\beta}\left(  r,\theta\right)  \theta^{\alpha}\otimes
\theta^{\beta}, \label{appendix11}%
\end{equation}
where, since the metric is infinitesimally Euclidean,
\begin{equation}
\sigma_{\alpha\beta}\left(  r,\theta\right)  =r^2\delta_{\alpha\beta}+o\left(
r^{2}\right)  \text{, as }r\searrow0. \label{appendix13}%
\end{equation}
We shall show that%
\[
\sigma_{\alpha\beta}\left(  r,\theta\right)  =r^{2}\delta_{\alpha\beta}.
\]
Since $\left(  [0,+\infty)\times S^{m-1},dr\otimes dr+r^{2}
\sum_\alpha \theta^\alpha\otimes \theta^\alpha\right)$ is isometric to $\mathbb{R}^{m}$
the proof will be completed.\medskip

\textit{Step 4. }Let $L_{\nabla r}$ denote the Lie derivative in the radial
direction $\nabla r$. We have%
\begin{equation}
\frac{\partial}{\partial r}\sigma_{\alpha\beta}=L_{\nabla r}\left\langle
,\right\rangle \left(  E_{\alpha},E_{\beta}\right)  =2 \mathrm{Hess}\left(
r\right)  \left(  E_{\alpha},E_{\beta}\right)  . \label{appendix12}%
\end{equation}
On the other hand, in view of (\ref{appendix7}), $\nabla r=\nabla f/\left\vert
\nabla f\right\vert $. \ Whence, using equation (\ref{appendix1}) we deduce
that, for every $E_{\alpha},E_{\beta}\in\nabla r^{\bot}$,%
\begin{equation}
\mathrm{Hess}\left(  r\right)  \left(  E_{\alpha},E_{\beta}\right)
=\left\langle D_{E_{\alpha}}\frac{\nabla f}{\left\vert \nabla f\right\vert
},E_{\beta}\right\rangle =\frac{1}{r}\sigma_{\alpha\beta}. \label{appendix15}%
\end{equation}
Combining (\ref{appendix13}), (\ref{appendix12}) and (\ref{appendix15}) we
conclude that the coefficients $\sigma_{\alpha\beta}$ solve the asymptotic
Cauchy problem%
\[
\left\{
\begin{array}
[c]{l}%
\dfrac{\partial\sigma_{\alpha\beta}}{\partial r}=\dfrac{2}{r}\sigma
_{\alpha\beta}\\
\sigma_{\alpha\beta}\left(  r,\theta\right)  =r^2 \delta_{\alpha\beta}+o\left(
r^{2}\right)  \text{, as }r\searrow0.
\end{array}
\right.
\]
Integrating finally gives%
\[
\sigma_{\alpha\beta}\left(  r,\theta\right)  =r^{2}\delta_{\alpha\beta},
\]
as desired.

\end{document}